\documentclass[a4paper,12pt,intlimits,oneside]{amsart}
\usepackage[utf8]{inputenc}
\usepackage{amsmath}
\usepackage{amsthm}
\usepackage{latexsym}
\usepackage{amssymb}
\usepackage{xcolor}
\usepackage{dsfont}
\usepackage{mathrsfs}
\usepackage[colorlinks=true]{hyperref}
\numberwithin{equation}{section}
\numberwithin{figure}{section}
\def\expo_#1{{\rm e}^{#1}}

\def\R{{\mathbb R}}
\def\C{{\mathbb C}}

\def\Z{{\mathbb Z}}

\def\1{1\!{\rm l}}

\def\build#1_#2^#3{\mathrel{\mathop{\kern 0pt#1}\limits_{#2}^{#3}}}
\def\td_#1,#2{\mathrel{\mathop{\build\longrightarrow_{#1\rightarrow #2}^{}}}}

\newcommand{\ben}{\begin{equation}}
\newcommand{\een}{\end{equation}}
\newcommand{\beno}{\begin{eqnarray*}}
\newcommand{\eeno}{\end{eqnarray*}}

\newtheorem{theorem}{Theorem}

\newtheorem{proposition}{Proposition}
\newtheorem{lemma}{Lemma}
\newtheorem{remark}{Remark}

\date{}
\title{UG}
\newcounter{rea}
\newcounter{reb}
\newcounter{res}
\title[Boundedness of the Bergman projection ]{Boundedness of the Bergman projection on some weighted mixed norm Lebesgue spaces of the upper-half space}
\author{}
\author[J.M. Tanoh Dje, Felix Ofori and Beno\^it. F.  Sehba]{Jean$-$Marcel Tanoh Dje, Felix Ofori and Beno\^it F. Sehba}
\address{Unit\'e de Recherche et d'Expertise Num\'erique, 
Universit\'e virtuelle de C\^ote d'Ivoire, Abidjan.}
\email{{\tt tanoh.dje@uvci.edu.ci}}
\address{Department of Mathematics, University of Ghana,  P.O. Box L.G 62 Legon, Accra, Ghana.}
\email{{\tt fofori006@st.ug.edu.gh}}

\address{Department of Mathematics, University of Ghana,  P.O. Box L.G 62 Legon, Accra, Ghana.}
\email{{\tt bfsehba@ug.edu.gh}}

\begin{document}
\maketitle
\begin{abstract}
  In this paper, we prove the boundedness of the Bergman projection on weighted mixed norm spaces of the upper-half space for some weights that are constructed using the logarithm function and growth functions. Our necessary and sufficient condition is the same as in the unweighted case, that is it involves only the parameters and not the weight. We then provide some applications in terms of dual and derivative characterization, and an atomic decomposition of the corresponding Bergman spaces. 
\end{abstract}
\section{Introduction}
\medskip
We put ourselves in the upper-half space $\C_+:=\{z\in\C: \Im mz>0\}$. We denote by $\mathcal H(\C_+)$ the set of holomorphic functions in $\C_+$. Let $\alpha>-1$, and let $1\le p,q<\infty$. For $\Omega$ a weight (positive function) on $(0,\infty)$, we denote by $L_{\Omega,\alpha}^{p,q}(\C_+)$ the set of all measurable functions $f$ on $\mathbb{C}_+$ such that
$$\Vert f\Vert_{\Omega,\alpha,p,q}^q:=\int_0^\infty\left(\int_{\mathbb{R}}|f(x+iy)|^pdx\right)^{\frac qp}\Omega(y)y^\alpha dy<\infty.$$
In particular, when $\Omega(y)=1$ and  $\alpha>-1$, we use the notations $L_\alpha^{p,q}(\mathbb{C}_+)=L^{p,q}(\mathbb{C}_+,(\Im mz)^\alpha dV(z))$ for the corresponding space and $\|\cdot\|_{\alpha,p,q}$ for the corresponding norm.
\medskip

 For $1\le p,q<\infty$ and $\alpha>-1$,
the mixed norm Bergman space $A_{\Omega,\alpha}^{p,q}(\mathbb{C}_+)$ is the intersection $L_{\Omega,\alpha}^{p,q}(\mathbb{C}_+)\cap \mathcal H(\C_+)$. We note that depending on the weight, this space can be trivial.
\medskip 

When $p=q$, the above space is just denoted $L_{\Omega,\alpha}^{p}(\C_+)$, and its norm $\Vert \cdot\Vert_{\Omega,\alpha,p}$.

For simplicity, we will be using the notation $$dV_\alpha(z)=\left(\Im m(z)\right)^\alpha dV(z)=y^\alpha dxdy\quad \textrm{for}\quad z=x+iy.$$ 
\medskip

For $\alpha>-1$, the Bergman projection $P_\alpha$ is the orthogonal projection from $L_\alpha^2(\mathbb{C}_+)$ into its closed subspace $A_\alpha^2(\mathbb{C}_+)$. This operator is  given by $$P_\alpha f(z)=\int_{\mathbb{C}_+}K_\alpha(z,w)f(w) dV_\alpha(w)$$
where $$K_\alpha(z,w)=\frac{c_\alpha}{\left(z-\bar{w}\right)^{2+\alpha}}$$
is the Bergman kernel, and $c_\alpha$ is a constant that depends only on $\alpha$.
\medskip
We denote by $P_\alpha^+$, the operator with kernel $\left|K_\alpha(\cdot,\cdot)\right|$.
\medskip

It is a well known and easy to verify fact that the Bergman projection $P_\alpha$ (and equivalently $P_\alpha^+$) extends to a bounded projection on $L_\alpha^p(\mathbb{C}_+)$ for any $1<p<\infty$. The problem of characterizing the weights $\omega$ such that $P_\alpha$ is bounded on $L^p(\mathbb{C}_+, \omega(z)dV_\alpha(z))$ was considered and solved by D. B\'ekoll\'e and A. Bonami in 1978 and 1982 (see \cite{Be1978,BeBo}). They proved that a necessary and sufficient condition for this to hold, is that the weight $\omega$ should satisfy the following condition also known as the $B_p$-condition:
\begin{equation}\label{eq:Bpcondition}
  [\omega]_p:=\sup_{I}\left(\frac{1}{|Q_I|_\alpha}\int_{Q_I}\omega(z)dV_\alpha(z)\right)\left(\frac{1}{|Q_I|_\alpha}\int_{Q_I}\left(\omega(z)\right)^{-\frac{p'}{p}}dV_\alpha(z)\right)^{\frac{p}{p'}}<\infty.  
\end{equation}
Here the supremum is taken over all subintervals $I$ of $\mathbb{R}$, $Q_I$ is the square above $I$, and $pp'=p'+p$.
\medskip


Obviously, this condition is difficult to verify for general weights. Moreover, this result doesn't say anything about the case of mixed norm spaces defined above. 
\medskip

Our aim in this paper is to give an answer to the problem of boundedness of the Bergman projection on weighted mixed norm spaces for an appropriate family of weights on $(0,\infty)$. The particularity here is that for the weights considered here, the condition obtained is the same as in the unweighted case (see \cite{BanSeh}).

Let us defined our weights. We recall that a growth function is any function 
$\Phi: \mathbb{R}_{+}\longrightarrow \mathbb{R}_{+}$, not identically null, that is continuous, nondecreasing and onto. 
\vskip .2cm
We observe that if $\Phi$ is a growth function, then  $\Phi(0)=0$ and $\lim_{t \to \infty}\Phi(t)=\infty$.
\medskip

Let  $p>0$ be a real and $\Phi$  a growth function. We say that  $\Phi$ is of upper-type (resp. lower-type) $p>0$ if there exists a constant
$C_{p}> 0$ such that for all $t\geq 1$ (resp. $0< t\leq 1$), 
\begin{equation}\label{eq:sui8n}
\Phi(st)\leq C_{p}t^{p}\Phi(s),~~\forall~s>0.\end{equation}
We denote by $\mathscr{U}^{p}$ (resp. $\mathscr{L}_{p}$) the set of all growth functions of upper-type $p \geq 1$ (resp. lower-type $0< p\leq 1$) such that the function 
$t\mapsto \frac{\Phi(t)}{t}$ is non decreasing (resp. non-increasing) on $\mathbb{R}_{+}^{*}= \mathbb{R}_{+}\backslash\{0\}$. We put
$   \mathscr{U}:=\bigcup_{p\geq 1}\mathscr{U}^{p}$ (resp. $\mathscr{L}:=\bigcup_{0< p\leq 1}\mathscr{L}_{p}$).
\medskip

The weights we consider here are powers of the following function
\begin{equation}\label{eq:formofweight}
    \mathcal{T}_\Phi^{\vec{\varepsilon}}(t)=1+\varepsilon_1\ln_+\left(\Phi\left(\frac{1}{t}\right)\right)+\varepsilon_2\ln_+\left(\Phi(t)\right), t>0.
\end{equation}
Here $\vec{\varepsilon}=(\varepsilon_1,\varepsilon_2)\in\{0,1\}^2$, and
$\Phi$ is a growth function. 

In fact in the upper half plane, these weights are taken as powers of
$$\mathcal{T}_\Phi^{\Vec{\varepsilon}}(z)=1+\varepsilon_1\ln_+\left(\Phi\left(\frac{1}{\Im mz}\right)\right)+\varepsilon_2\ln_+\left(\Phi(|z|)\right), (\varepsilon_1,\varepsilon_2)\in\{0,1\}^2.$$
When $\Phi(t)=t$, and $(\varepsilon_1,\varepsilon_2)=(1,1)$, the corresponding weight $\omega(z)$ is exactly the one that appears in the study of Bloch functions of the upper half plane (see \cite{BGS1}). 
\medskip 

 Recall that  a function $
f$ is a Bloch function if \begin{equation} \label{bloch-intro}
    \Im m(z)|f'(z)|\leq C.
\end{equation}
Let us put $$\|f\|_*:=|f(i)|+\sup_{z\in \C_+}\Im m(z)|f'(z)|.$$ Then one has (see \cite{BGS1}) that there is a constant $C>0$ such that for any Bloch function $f$, 
$$|f(z)|\leq C\|f\|_*\left(\ln\left((e+\frac{1}{\Im mz}\right)+\ln\left(e+|z|\right)\right), z\in \C_+.$$
\medskip
Our main result is as follows.

\begin{theorem}\label{thm:main1}
Let $1\leq p,q<\infty$, $\alpha, \beta>-1$. Assume that $\Phi\in \mathscr{L}\cup\mathscr{U}$. Put $\omega=\mathcal{T}_\Phi^{\Vec{\varepsilon}}$, and let $k\in\mathbb{R}$. Then the following assertions are equivalent.
\begin{enumerate}
    \item[(a)] The operator $P_\beta$ is bounded on  $L^{p,q}(\C_+, \omega^{k}(\Im mz)dV_\alpha(z))$.
    \item[(b)] The operator $P_\beta^+$ is bounded on  $L^{p,q}(\C_+, \omega^{ k}(\Im mz)dV_\alpha(z))$.
    \item[(c)] The parameters satisfy
    \begin{equation}\label{eq:paramrelation}
        \alpha+1<q(\beta+1).
    \end{equation}
    .
    
\end{enumerate}
\end{theorem}
\medskip

It is known that the boundedness of $P_\beta^+$ on $L_\alpha^{q}(\mathbb{C}_+)$ is equivalent to the boundedness of a corresponding Hilbert-type operator of the positive line (see for example \cite{BanSeh,Sehba}). The same observation can be made here. The Hilbert-type operator $H_\beta$ is defined by
$$H_\beta f(x)=\int_0^\infty\frac{f(y)y^\beta dy}{\left(x+y\right)^{1+\beta}}.$$
Let $\Omega$ be a weight on $(0,\infty)$, and $1\le p<\infty$. We denote by $L_\Omega^p((0,\infty))$ the set of all measurable functions $f$ such that
$$\Vert f\Vert_{\Omega,p}^p:=\int_0^\infty |f(x)|^p\Omega(x)dx<\infty.$$
In particular, when $\Omega(x)=x^\alpha$, $\alpha\in\mathbb{R}$, we use the notations $L_\alpha^p((0,\infty))=L^p((0,\infty), x^\alpha dx)$ for the corresponding space and $\|\cdot\|_{\alpha,p}$ for the corresponding norm.
\medskip

It is a well known that the operator $H_\beta$ is bounded on $L_\alpha^p((0,\infty))$ for any $1<p<\infty$ if and only if the relation $1+\alpha<p(1+\beta)$ holds (see for example \cite{BanSeh}). We extend this result to the weighted space $L^p((0,\infty), \omega(x)x^\alpha dx)$ for the above family of weights.
Our main results on the weighted boundedness of $H_\beta$ are given in the following.
\begin{theorem}\label{thm:main2}
Let $1\leq p<\infty$, $\alpha, \beta>-1$. Assume that $\Phi\in \mathscr{L}\cup\mathscr{U}$. Put $\omega=\mathcal{T}_\Phi^{\Vec{\varepsilon}}$, and let $k\in\mathbb{R}$. Then the following assertions are equivalent.
\begin{enumerate}
    \item[(a)] The operator $H_\beta$ is bounded on  $L^p((0,\infty), \omega^{k}(y)y^\alpha dy)$.
    \item[(b)] The parameters satisfy
    \begin{equation}\label{eq:paramrelation1}
        \alpha+1<p(\beta+1).
    \end{equation}
\end{enumerate}
\end{theorem}

\medskip 

To prove the above result, we will prove a kind of weighted Forelli-Rudin type estimate adapted to our setting and then combine it with a Schur's lemma in the case $p>1$. 
\medskip
\begin{remark}
In general, to study the boundededness of operators with kernel on weighted Lebesgue spaces, one needs to develop a weighted theory as the ones of Muckenhoupt weights for Calder\'on-Zygmund operators (see for example \cite{GCRF}) or of the B\'ekoll\'e-Bonami weights for the Bergman operator (see \cite{Be1978,BeBo}).Our family of weights allows us to avoid this approach here and to reduce the question to a relation between the parameters $\alpha,\beta$ and the exponent $p$. The weighted theory for the operator $H_\beta$ (for general weights) will be considered elsewhere. 

\end{remark}
\medskip

The choice of the weights in this paper was inspired by \cite{BGS1} which also serves as motivation for this paper. Among our applications, we have an atomic decomposition of the above Bergman spaces which generalizes the work of J. Gonessa on the same topic in \cite{Gonessa}.
\medskip

The paper is organized as follows. In Section 2, we prove the result on the Hilbert-type operators. The boundedness of the Bergman projection on mixed norm spaces is proved in Section 3. In the last section, we provide some applications of our results.

As usual, we use the  the notation 
$A\lesssim B$ (respectively $A\gtrsim B$) whenever there exists a uniform constant $C>0$ such that $A\le CB$ (respectively $A\ge CB$). The notation $A\simeq B$ means that $A\lesssim B$ and $A\gtrsim B$.

\section{Boundedness of the Hilbert-type operator $H_\beta$}
In this section, we consider the boundedness of the operator $H_\beta$ on $L^p((0,\infty), \omega^{k}(y)y^\alpha dy)$. 
\subsection{Weighted estimates on the positive line}
In the following, we give estimates of the measure of any interval $(0,t)$ with respect to $\omega^{k}(y)y^\beta dy$.
\begin{lemma}\label{lem:omegainterval1}
Let $\beta>-1$. Suppose that $\Phi\in \mathscr{L}\cup\mathscr{U}$. Put $\omega=\mathcal{T}_\Phi^{\Vec{\varepsilon}}$, and let $k>0$. Then the following holds. 
\begin{equation}\label{eq:omegainterval1}
    \int_0^t\omega^{-k}(y)y^\beta dy\simeq \omega^{-k}(t)\,t^{1+\beta}.
\end{equation}
\end{lemma}
\begin{proof}
We restrict ourselves to the case $\Vec{\varepsilon}=(\varepsilon_1,\varepsilon_2)=(1,1)$ as the other cases follow the same way. Recall that in this case,
$$\omega(y)=1+\ln_+\left(\Phi\left(\frac{1}{y}\right)\right)+\ln_+\left(\Phi(y)\right),$$
and this is comparable to
$$\ln\left(e+\Phi\left(\frac{1}{y}\right)\right)+\ln\left(e+\Phi(y)\right).$$
We start by proving the upper estimate. If $0<t\le 1$, then the estimate above in (\ref{eq:omegainterval1}) follows directly from the fact that $\omega$ is nonincreasing in this interval. 
\medskip

More generally, let  $t>0$. Then using that $\Phi$ satisfies (\ref{uppertype}) or (\ref{lowertype}), and the notation
\begin{equation}\label{eq:pPhi} p_\Phi= \left\{\begin{array}{lcr}q & \mbox{if} & \Phi\in \mathscr{U}^q\\ 1 & \mbox{if} & \Phi\in \mathscr{L}_p.\end{array},\right.\end{equation}
we obtain
\begin{eqnarray*}
I &:=& \int_0^t\omega^{-k}(y)y^\beta dy\\ &=& \sum_{n=0}^\infty \int_{2^{-n-1}t}^{2^{-n}t}\frac{y^\beta dy}{\left[\ln\left(e+\Phi\left(\frac{1}{y}\right)\right)+\ln\left(e+\Phi(y)\right)\right]^k}\\ &\leq& \frac{(2^{1+\beta}-1)}{1+\beta}t^{1+\beta}\sum_{n=0}^\infty\frac{2^{-(n+1)(1+\beta)}}{\left[\ln\left(e+\Phi\left(\frac{2^n}{t}\right)\right)+\ln\left(e+\Phi\left(\frac{t}{2^{n+1}}\right)\right)\right]^k}\\ &\lesssim&  t^{1+\beta}\sum_{n=0}^\infty\frac{2^{-(n+1)(1+\beta)}\ln^k\left(2^{(n+1)p_\Phi}\right)}{\left[\ln\left(e+\Phi\left(\frac{1}{t}\right)\right)+\left[\ln\left(2^{(n+1)p_\Phi}\right)+\ln\left(e+\Phi\left(\frac{t}{2^{n+1}}\right)\right)\right]\right]^k}\\ &\lesssim& \frac{t^{1+\beta}\left(p_\Phi\ln 2\right)^k}{\omega^k(t)}\sum_{n=0}^\infty 2^{-(n+1)(1+\beta)}(n+1)^k\\ &\lesssim& \frac{t^{1+\beta}}{\omega^k(t)}.
\end{eqnarray*}
The lower inequality follows by considering the integral over $(\frac{t}{2},t)$. The proof is complete.
\end{proof}
We also have the following.
\begin{lemma}\label{lem:omegainterval2}
Let $\beta>-1$. Suppose that $\Phi\in \mathscr{L}\cup\mathscr{U}$. Put $\omega=\mathcal{T}_\Phi^{\Vec{\varepsilon}}$, and let $k>0$. Then the following 
\begin{equation}\label{eq:omegainterval2}
    \int_0^t\omega^{k}(y)y^\beta dy\simeq \omega^{k}(t)\,t^{1+\beta}.
\end{equation}
\end{lemma}
\begin{proof}
The lower estimate is obtained as above. For the upper estimate, we again restrict to the weight
$$\omega(y)=\ln\left(e+\Phi\left(\frac{1}{y}\right)\right)+\ln\left(e+\Phi(y)\right).$$
If $t\geq 1$, then the upper estimate in (\ref{eq:omegainterval2}) follows directly from the fact that $\omega$ is nonincreasing in this range.
\medskip

More generally, let $t>0$. We obtain
\begin{eqnarray*}
I &:=& \int_0^t\omega^{k}(y)y^\beta dy\\ &=& \sum_{n=0}^\infty \int_{2^{-n-1}t}^{2^{-n}t}\left[\ln\left(e+\Phi\left(\frac{1}{y}\right)\right)+\ln\left(e+\Phi(y)\right)\right]^ky^\beta dy\\ &\lesssim& (2^{1+\beta}-1)t^{1+\beta}\sum_{n=0}^\infty2^{-(n+1)(1+\beta)}\left[\ln\left(e+\Phi\left(\frac{2^{n+1}}{t}\right)\right)+\ln\left(e+\Phi(2^{-n}t)\right)\right]^k\\ &\lesssim&  t^{1+\beta}\sum_{n=0}^\infty 2^{-(n+1)(1+\beta)}\left[(n+1)p_\Phi\ln2 +\ln\left(e+\Phi\left(\frac{1}{t}\right)\right)+\ln\left(e+\Phi(t)\right)\right]^k\\ &\leq& t^{1+\beta}\left(p_\Phi\right)^k\omega^k(t)\sum_{n=0}^\infty 2^{-(n+1)(1+\beta)}(n+1)^k\\ &\lesssim& t^{1+\beta}\omega^k(t).
\end{eqnarray*}
The proof is complete.
\end{proof}
We have the following estimate.
\begin{lemma}\label{lem:FR1}
Let $\alpha>0$, and $\beta>-1$. Suppose that $\Phi\in \mathscr{L}\cup\mathscr{U}$. Put $\omega=\mathcal{T}_\Phi^{\Vec{\varepsilon}}$, and let $k>0$. Then the following holds
\begin{equation}\label{eq:FR1}
    \int_0^\infty\frac{\omega^{-k}(y)y^\beta dy}{\left(x+y\right)^{1+\alpha+\beta}}\simeq \omega^{-k}(x)\,x^{-\alpha}, x>0.
\end{equation}
\end{lemma}
\begin{proof}
Again we only present the proof for $\Vec{\varepsilon}=(1,1)$. For $x>0$ fixed, let $I_0=(0,x]$. For $j\in \mathbb{N}$, let $I_j=(0,2^jx]$. Then define $J_0=I_0$, and for $j\geq 1$, $J_j=I_j\setminus I_{j-1}$.
\medskip

We obtain using Lemma \ref{lem:omegainterval1} that
\begin{eqnarray*}
L &:=&  \int_0^\infty\frac{\omega^{-k}(y)y^\beta dy}{\left(x+y\right)^{1+\alpha+\beta}}\\ &=& \sum_{j=0}^\infty  \int_{J_j}\frac{\omega^{-k}(y)y^\beta dy}{\left(x+y\right)^{1+\alpha+\beta}}\\ &\leq& x^{-(1+\alpha+\beta)}\sum_{j=0}^\infty 2^{-(j-1)(1+\alpha+\beta)} \int_{I_j}\omega^{-k}(y)y^\beta dy\\ &\lesssim& x^{-(1+\alpha+\beta)}\sum_{j=0}^\infty 2^{-(j-1)(1+\alpha+\beta)}\frac{\left(2^jx\right)^{1+\beta}}{\omega^k(2^jx)}.
\end{eqnarray*}
If $x\geq 1$, then we directly have that $\omega(2^jx)>\omega(x)$. If $0<x<1$, then for any $j>0,$ we obtain
\begin{eqnarray*}
\omega(2^jx)&=& \ln\left(e+\Phi\left(2^jx\right)\right)+\ln\left(e+\Phi\left(\frac{1}{2^jx}\right)\right)\\ &\gtrsim& \frac{\ln\left(e+\Phi(x)\right)+\ln\left(2^{jp_\Phi}\right)+\ln\left(e+\Phi\left(\frac{1}{2^jx}\right)\right)}{\ln\left(2^{jp_\Phi}\right)}\\ &\gtrsim& \frac{\ln\left(e+\Phi(x)\right)+\ln\left(e+\Phi\left(\frac{1}{x}\right)\right)}{\ln\left(2^{jp_\Phi}\right)}.
\end{eqnarray*}
Hence we finally obtain
\begin{eqnarray*}
L &\lesssim& \omega^{-k}(x)\,x^{-\alpha}\sum_{j=0}^\infty 2^{-(j-1)\alpha}j^k\\ &\lesssim& \omega^{-k}(x)\,x^{-\alpha}.
\end{eqnarray*}
The lower estimate is obtained by taking the integral from $\frac{x}{2}$ to $x$. The proof is complete.
\end{proof}
We have the following.
\begin{lemma}\label{lem:FR2}
Let $\alpha>0$, and $\beta>-1$. Suppose that $\Phi\in \mathscr{L}\cup\mathscr{U}$. Put $\omega=\mathcal{T}_\Phi^{\Vec{\varepsilon}}$, and let $k>0$. Then the following holds 
\begin{equation}\label{eq:FR2}
    \int_0^\infty\frac{\omega^{k}(y)y^\beta dy}{\left(x+y\right)^{1+\alpha+\beta}}\simeq \omega^{k}(x)\,x^{-\alpha}, x>0.
\end{equation}
\end{lemma}
\begin{proof}
The proof is essentially the same as above. One only needs the upper estimate of $\omega(2^jx)$, $j>0$. If $x>0$, then 
\begin{eqnarray*}\omega(2^jx) &\lesssim& \ln\left(e+\Phi\left(\frac{1}{x}\right)\right)+\ln\left(2^{jp_\Phi}\right)+\ln\left(e+\Phi(x)\right)\\ &=& \ln\left(e+\Phi\left(\frac{1}{x}\right)\right)+jp_\Phi\ln 2+\ln\left(e+\Phi(x)\right)\\ &\lesssim& j\omega(x).\end{eqnarray*}
The estimates then follow as in the previous lemma.
\end{proof}
\subsection{Proofs of Theorem \ref{thm:main2}}
We start by considering the case $p=1$. We prove the following.
\begin{theorem}\label{thm:q1}
Let  $\alpha, \beta>-1$. Assume that $\Phi\in \mathscr{L}\cup\mathscr{U}$. Put $\omega=\mathcal{T}_\Phi^{\Vec{\varepsilon}}$, and let $k\in\mathbb{R}$. Then the following assertions are equivalent.
\begin{enumerate}
    \item[(a)] The operator $H_\beta$ is bounded on  $L^1((0,\infty), \omega^{ k}(y)y^\alpha dy)$.
    \item[(b)] The parameters satisfy
    $$\beta>\alpha.$$
\end{enumerate}
\end{theorem}
\begin{proof}
That (b) implies (a) follows from Fubini's theorem and Lemma \ref{lem:FR1} and Lemma \ref{lem:FR2}. 
\medskip 

Conversely, suppose that $H_\beta$ is bounded on  $L^1((0,\infty), \omega^{-k}(y)y^\alpha dy)$, $k>0$. Then its adjoint $H_\beta^*$ defined by
$$H_\beta^*f(x)=\omega^k(x)x^{\beta-\alpha}\int_0^\infty\frac{f(y)\omega^{-k}(y)y^\alpha dy}{\left(x+y\right)^{1+\beta}}$$
is bounded on $L^\infty((0,\infty))$. Take $f(y)=\chi_{[1,2]}(y)$. Then
$$H_\beta^*f(x)\simeq\frac{\omega^k(x)x^{\beta-\alpha}}{\left(x+1\right)^{1+\beta}}.$$
Clearly, for $0<x<1$, one has that $H_\beta^*f(x)\gtrsim x^{\beta-\alpha}.$
Hence the boundedness of $H_\beta^*f(x)$ implies that $x^{\beta-\alpha}$ is bounded on $(0,1)$. Thus, necessarily, $\beta-\alpha>0$. 
\medskip

Now, suppose that $H_\beta$ is bounded on  $L^1((0,\infty), \omega^{k}(y)y^\alpha dy)$, $k>0$ i.e there exists a constant $C>0$ such that for any $f\in L^1((0,\infty), \omega^{k}(y)y^\alpha dy)$,

\begin{equation}\label{eq:testhilbert1}
    \int_0^\infty |H_\beta f(x)|\omega^{k}(x)x^\alpha dx\le C\int_0^\infty |f(x)|\omega^{k}(x)x^\alpha dx
\end{equation}
Take $f(y)=\chi_{[1,2]}(y)$. Then $$H_\beta f(x)\simeq \frac{1}{(x+1)^{1+\beta}}$$ and 
$$\int_0^\infty |f(x)|\omega^{k}(x)x^\alpha dx\simeq 1.$$
Taking this into (\ref{eq:testhilbert1}), we obtain
$$ \int_0^\infty\frac{x^\alpha}{(x+1)^{1+\beta}}dx\le\int_0^\infty\frac{\omega^k(x)x^\alpha}{(x+1)^{1+\beta}}dx<\infty.$$
This necessarily implies that $1+\beta>\alpha+1$.
The proof is complete.
\end{proof}

We next consider the case $p>1$.
\begin{theorem}\label{thm:qlarge}
Let $1<p<\infty$, and  $\alpha, \beta>-1$. Assume that $\Phi\in \mathscr{L}\cup\mathscr{U}$. Put $\omega=\mathcal{T}_\Phi^{\Vec{\varepsilon}}$, and let $k\in\mathbb{R}$. Then the following assertions are equivalent.
\begin{enumerate}
    \item[(a)] The operator $H_\beta$ is bounded on  $L^p((0,\infty), \omega^{k}(y)y^\alpha dy)$.
    \item[(b)] The parameters satisfy
    (\ref{eq:paramrelation1}).
\end{enumerate}
\end{theorem}
\begin{proof}
We restrict to the case of $\omega^{-k}$, $k>0$ as the other case follows from the same ideas. Recall that $$H_\beta f(x)=\int_0^\infty\frac{f(y)y^\beta dy}{\left(x+y\right)^{1+\beta}}.$$
We note that the kernel of $H_\beta$ with respect to the measure $\omega^{-k}(y)y^\alpha dy$ is $$K(x,y)=\frac{\omega^k(y)y^{\beta-\alpha}}{\left(x+y\right)^{1+\beta}}.$$
Assume that (\ref{eq:paramrelation1}) holds. By the Schur's lemma (see \cite{Grafakos}), to prove that $H_\beta$ is bounded on 
$L^p((0,\infty), \omega^{-k}(y)y^\alpha dy)$, it suffices to find a positive function $\phi$ such that
\begin{equation}\label{eq:schur1}
    \int_0^\infty K(x,y)\phi^{p'}(y)\omega^{-k}(y)y^\alpha dy\leq C\phi^{p'}(x)
\end{equation}
and
\begin{equation}\label{eq:schur2}
    \int_0^\infty K(x,y)\phi^{p}(x)\omega^{-k}(x)x^\alpha dx\leq C\phi^{p}(y).
\end{equation}
Let us take $\phi(t)=t^{-\frac{\alpha+1}{pp'}}$. Then using Lemma \ref{lem:FR1}, we obtain
\begin{eqnarray*}
L &:=& \int_0^\infty K(x,y)\phi^{p'}(y)\omega^{-k}(y)y^\alpha dy\\ &=& \int_0^\infty\frac{y^{\beta-\frac{\alpha+1}{p}} dy}{\left(x+y\right)^{1+\beta}}\\ &\simeq& x^{-\frac{\alpha+1}{p}}=\phi^{p'}(x).
\end{eqnarray*}
In the same way, we have
\begin{eqnarray*}
L &:=& \int_0^\infty K(x,y)\phi^{p}(x)\omega^{-k}(x)x^\alpha dx\\ &=& \omega^k(y)y^{\beta-\alpha}\int_0^\infty\frac{\omega^{-k}(x)x^{-\frac{\alpha+1}{p'}+\alpha} dx}{\left(x+y\right)^{1+\beta}}\\ &\simeq& x^{-\frac{\alpha+1}{p'}}=\phi^{p}(x).
\end{eqnarray*}
This completes the proof of the sufficient part.
\medskip

Now suppose that $H_\beta$ is bounded on  $L^p((0,\infty), \omega^{-k}(y)y^\alpha dy)$. Then its adjoint $H_\beta^*$ defined by
$$H_\beta^*f(x)=\omega^k(x)x^{\beta-\alpha}\int_0^\infty\frac{f(y)\omega^{-1}(y)y^\alpha dy}{\left(x+y\right)^{1+\beta}}$$
is bounded on $L^{p'}((0,\infty), \omega^{-k}(y)y^\alpha dy)$. Take $f(y)=\chi_{[1,2]}(y)$. Then
$$H_\beta^*f(x)\simeq\frac{\omega^k(x)x^{\beta-\alpha}}{\left(x+1\right)^{1+\beta}}\geq \frac{\omega^{\frac{k}{p'}}(x)x^{\beta-\alpha}}{\left(x+1\right)^{1+\beta}}.$$
Thus that  $$\|H_\beta^*f\|_{L_{\omega^{-k},\alpha}^{p'}}\leq C\|f\|_{L_{\omega^{-k},\alpha}^{p'}}\simeq C$$ implies that $$\int_0^\infty \frac{x^{(\beta-\alpha)p'+\alpha}}{\left(x+1\right)^{(1+\beta)p'}}dx<\infty.$$ Thus, necessarily, $(\beta-\alpha)p'+\alpha>-1$. That is (\ref{eq:paramrelation1}) holds. The proof is complete.
\end{proof}
\begin{remark}
To prove the implication (a)$\Rightarrow$(b) in the case of $\omega^k$, $k>0$ in the above result, one can work directly with $H_\beta$ and test the boundedness on the function $f(y)=\chi_{[1,2]}(y)$ as in Theorem \ref{thm:q1}.
\end{remark}
\section{Boundedness of the Bergman projection}
\subsection{Proofs of Theorem \ref{thm:main1}}
We prove the various cases appearing in Theorem \ref{thm:main1}.
\begin{proposition}\label{prop:main1}
Let $1< p,q<\infty$, $\alpha, \beta>-1$. Assume that $\Phi\in \mathscr{L}\cup\mathscr{U}$. Put $\omega=\mathcal{T}_\Phi^{\Vec{\varepsilon}}$, and let $k\in\mathbb{R}$. Then the following assertions are equivalent.
\begin{enumerate}
    \item[(a)] The operator $P_\beta$ is bounded on  $L^{p,q}(\C_+, \omega^{ k}dV_\alpha)$.
    \item[(b)] The operator $P_\beta^+$ is bounded on  $L^{p,q}(\C_+, \omega^{k}dV_\alpha)$.
    \item[(c)] The parameters satisfy
    \begin{equation*}
        \alpha+1<q(\beta+1).
    \end{equation*}
    .
    
\end{enumerate}
\end{proposition}
\begin{proof}
The implication (b)$\Rightarrow$(a) is clear. Let us prove that (c)$\Rightarrow$(b) and (a)$\Rightarrow$(c).
\medskip

(c)$\Rightarrow$(b): Let write $f_y(x)=f(x+iy)$. Then it is easy to see (see for example \cite[Lemma 5.1]{BanSeh}) that 
 $$\left\|\left(P_\beta^+f\right)_y\right\|_{L^p(\mathbb{R})}\le C_{\beta}H_{\beta}(\|f_v\|_{L^p(\mathbb{R})})(y).$$  Thus the boundedness of $H_{\beta}$ on $L^{q}((0,\infty) \omega^{k})(y)y^\alpha dy$ implies the boundedness of  $P_\beta^+$ on $L^{p,q}(\mathbb{C}_+,\omega^{ k}dV_\alpha)$. Hence the implication  follows from this observation and Theorem \ref{thm:main2}.
\medskip

(a)$\Rightarrow$(c): Assume that $P_\beta$ is bounded on  $L^{p,q}(\C_+, \omega^{- k}(y)y^\alpha dxdy)$, $k>0$. Then its adjoint $P_\beta^*$ defined by 

$$P_\beta^* f(z)=\omega^k(\Im mz)\left(\Im mz\right)^{\beta-\alpha}\int_{\mathbb{C}_+}\frac{f(w)}{\left(z-\bar{w}\right)^{2+\beta}} \omega^{-k}(\Im m\,w)dV_\alpha(w)$$
is bounded on $L^{p',q'}(\C_+, \omega^{- k}(y)y^\alpha dxdy)$.
\vskip .2cm
Let us take $f(w)=\omega^k(\Im m\,w)\left(\Im m\,w\right)^{-\alpha}\chi_{B(i,1)}(w)$ where $B(i,1)$ is the euclidean ball centered at $i$ with radius $1$. Applying the mean value property, we obtain that
$$P_\beta^* f(z)= c\frac{\omega^k(\Im mz)\left(\Im mz\right)^{\beta-\alpha}}{\left(z+i\right)^{2+\beta}}.$$

Thus as $\omega(y)\geq 1$, we obtain that 
$$\|P_\beta^*f\|_{L^{p',q'}(\omega^{- k}dV_\alpha)}\le C\|f\|_{L^{p',q'}(\omega^{- k}dV_\alpha)}=C$$
implies that the function $z\longrightarrow \frac{\left(\Im mz\right)^{\beta-\alpha}}{\left(z+i\right)^{2+\beta}}$ belongs to $L^{p',q'}(\C_+,dV_\alpha)$. In particular, we have that
$$\int_0^1y^{(\beta-\alpha)q'+\alpha}dy\lesssim \int_0^1\left(\int_{-\frac 14}^{\frac 14}\frac{dx}{|x+i(y+1)|^{(2+\beta)p'}}\right)^{\frac{q'}{p'}}y^{(\beta-\alpha)q'+\alpha}dy<\infty.$$
Hence one necessarily has that $(\beta-\alpha)q'+\alpha>-1$ or equivalently, $\alpha+1<q(\beta+1)$.
\medskip

Next, suppose that $P_\beta$ is bounded on  $L^{p,q}(\C_+, \omega^{k}(y)y^\alpha dxdy)$, $k>0$. Then its adjoint $P_\beta^*$ defined by 

$$P_\beta^* f(z)=\omega^{-k}(\Im mz)\left(\Im mz\right)^{\beta-\alpha}\int_{\mathbb{C}_+}\frac{f(w)}{\left(z-\bar{w}\right)^{2+\beta}} \omega^{k}(\Im m\,w)dV_\alpha(w)$$
is bounded on $L^{p',q'}(\C_+, \omega^{ k}(y)y^\alpha dxdy)$.
\vskip .2cm
Let us take $f(w)=\omega^{-k}(\Im m\,w)\left(\Im m\,w\right)^{-\alpha}\chi_{B(i,1)}(w)$ where $B(i,1)$ is again the euclidean ball centered at $i$ with radius $1$. Applying the mean value property again, we obtain that
$$P_\beta^* f(z)= c\frac{\omega^{-k}(\Im mz)\left(\Im mz\right)^{\beta-\alpha}}{\left(z+i\right)^{2+\beta}}.$$

Thus
$$\|P_\beta^*f\|_{L^{p',q'}(\omega^{k}dV_\alpha)}\le C\|f\|_{L^{p',q'}(\omega^{ k}dV_\alpha)}=C$$
implies that 
\begin{equation}\label{eq:necesP}\int_0^1\omega^{-kq'}(y)y^{(\beta-\alpha)q'+\alpha}dy\lesssim \int_0^1\left(\int_{-\frac 14}^{\frac 14}\frac{dx}{|x+i(y+1)|^{(2+\beta)p'}}\right)^{\frac{q'}{p'}}\omega^{-kq'+k}(y)y^{(\beta-\alpha)q'+\alpha}dy<\infty.\end{equation}
To prove that necessarily, $(\beta-\alpha)q'+\alpha+1>0$, it suffices to show that the left hand side integral in (\ref{eq:necesP}) is not convergent if $(\beta-\alpha)q'+\alpha+1\leq 0$. In fact it suffices to prove this for $(\beta-\alpha)q'+\alpha+1= 0$.
\vskip .2cm
Suppose that $(\beta-\alpha)q'+\alpha+1= 0$. From (\ref{eq:necesP}) we have that
\begin{eqnarray*}
    \infty &>& \int_0^1\omega^{-kq'}(y)y^{-1}dy\\ &\gtrsim& \int_0^1\left(\ln\Phi\left(\frac{e}{y}\right)\right)^{-kq'}y^{-1}dy\\ &\gtrsim& \int_0^1\left(\frac{e}{y}\right)^{-p_\Phi kq'}y^{-1}dy.
\end{eqnarray*}
This is impossible since $\int_0^1y^{-1-p_\Phi kq'}dy$ is not convergent.
The proof is complete.
\end{proof}
In the case $q=1$, we have the following.
\begin{proposition}\label{prop:largep}
Let $1\le p<\infty$, $\alpha, \beta>-1$. Assume that $\Phi\in \mathscr{L}\cup\mathscr{U}$. Put $\omega=\mathcal{T}_\Phi^{\Vec{\varepsilon}}$, and let $k\in\mathbb{R}$. Then the following assertions are equivalent.
\begin{enumerate}
    \item[(a)] The operator $P_\beta$ is bounded on  $L^{p,1}(\C_+, \omega^{k}dV_\alpha)$.
    \item[(b)] The operator $P_\beta^+$ is bounded on  $L^{p,1}(\C_+, \omega^{k}dV_\alpha)$.
    \item[(c)] The parameters satisfy
    \begin{equation*}
        \alpha<\beta.
    \end{equation*}
       
\end{enumerate}
\end{proposition}
\begin{proof}
The implications (b)$\Rightarrow$(a) and (c)$\Rightarrow$(b) follow as above with the help of Theorem \ref{thm:q1}.
\medskip 

Reasoning as above, we obtain that the functions $z\longrightarrow \frac{\left(\Im mz\right)^{\beta-\alpha}}{\left(z+i\right)^{2+\beta}}$ and $z\longrightarrow \frac{\omega^{-k}(\Im mz)\left(\Im mz\right)^{\beta-\alpha}}{\left(z+i\right)^{2+\beta}}$ belong to $L^{p',\infty}(\mathbb{C}_+)$. This necessarily leads to $\beta>\alpha$.
\end{proof}
\section{Applications}
We provide in this section, three applications of our above results. For our purpose in this section, we need to recall the definition of the Hardy spaces. For $1\le p<\infty$, the Hardy space $H^p(\mathbb{C}_+)$ consists of all holomorphic functions $f$ on $\mathbb{C}_+$ such that
$$\|f\|_p:=\sup_{y>0}\left(\int_{\mathbb{R}}|f(x+iy)|^pdx\right)^{1/p}<\infty.$$
\subsection{Reproducing formulas and duality}
We prove reproducing formula only in the case of the weight $\omega^{-k}$ with $k>0$ as the corresponding spaces are larger than the usual mixed norm Bergman spaces that correspond to the case $\omega=1$ for which reproducing formulas are known (see for example \cite{BBPR,BBGNPR,Gonessa}).
\medskip

Put $\Phi\in \mathscr{L}\cup\mathscr{U}$. Put $\omega=\mathcal{T}_\Phi^{\Vec{\varepsilon}}$. For $f\in A^{p,q}_{\omega^{-k},\alpha}(\mathbb{C}_+)$, $k>0$, and for $x\in \mathbb{R}$ and $y>0$ given, defined $f_1$ and $f_2$ by
$$f_1(u+iv)=f(x+u+iv),\text{and}\quad f_2(u+iv)=f(y(u+iv)).$$
Then $$\|f_1\|_{A^{p,q}_{\omega^{-k},\alpha}}=\|f\|_{A^{p,q}_{\omega^{-k},\alpha}}$$
and $$\|f_2\|_{A^{p,q}_{\omega^{-k},\alpha}}=y^{-\frac{1+\alpha}{q}-\frac {1}p}\|f\|_{A^{p,q}_{\omega_{y^{-1}}^{-k},\alpha}}$$
where $\omega_{y^{-1}}(v)=\omega(y^{-1}v)$. Moreover, we have the estimate
\begin{equation}
  \omega^{-1}(y^{-1}v)\lesssim \left(1+\epsilon_1\ln_+\left(\frac 1y\right)+\epsilon_2\ln_+(y)\right)\omega^{-1}(v).  
\end{equation}
For simplicity, we put $$\omega_0(y)=1+\epsilon_1\ln_+\left(\frac 1y\right)+\epsilon_2\ln_+(y).$$
We start this part by the following result which follows as in \cite[Proposition 1.3.4]{Gonessa} (see also \cite{BBGNPR}) with the help of the above observations.
\begin{proposition}\label{prop:basicineq}
\begin{itemize}
Let $1\leq p,q<\infty$, and  $\alpha>-1$. Assume that $\Phi\in \mathscr{L}\cup\mathscr{U}$. Put $\omega=\mathcal{T}_\Phi^{\Vec{\varepsilon}}$, and let $k>0$. Then the following assertions are satisfied.
\item[(i)] There exists a constant $C>0$ such that for all $x+iy\in \mathbb{C}_+$ and all $f\in A^{p,q}_{\omega^{-k},\alpha}(\mathbb{C}_+)$, the following inequality holds.
\begin{equation}\label{eq:pointwiserho}
|f(x+iy)|\le C\left(\omega_0(y)\right)^{\frac {k}q}y^{-\frac{1+\alpha}{q}-\frac {1}p}\|f\|_{A^{p,q}_{\omega^{-k},\alpha}}.
\end{equation}
\item[(ii)]There exists a constant $C>0$ such that for all $y\in (0,\infty)$ and all $f\in A^{p,q}_{\omega^{-k},\alpha}(\mathbb{C}_+)$,
\begin{equation}\label{eq:hardyrrho}
\|f_y\|_{p}:=\|f(\cdot+iy)\|_{L^p(\mathbb{R})}\le C\left(\omega_0(y)\right)^{\frac {k}q}y^{-\frac{1+\alpha}q}\|f\|_{A^{p,q}_{\omega^{-k},\alpha}}.
\end{equation}
\item[(iii)] For all $f\in A^{p,q}_{\omega^{-k},\alpha}(\mathbb{C}_+)$, and for all $y>0$, the following holds:
$$\lim_{|x|\rightarrow\infty}|f(x+iy)|=0=\lim_{|x|+y\rightarrow\infty, y>y_0}|f(x+iy)|.$$
\end{itemize}

\end{proposition}
We then obtain the following.
\begin{proposition}\label{prop:densityomega}
Let $1\leq p,q<\infty$, and  $\alpha>-1$. Assume that $\Phi\in \mathscr{L}\cup\mathscr{U}$. Put $\omega=\mathcal{T}_\Phi^{\Vec{\varepsilon}}$, and let $k>0$. Let $f\in A^{p,q}_{\omega^{-k},\alpha}(\mathbb{C}_+)$. Then the following assertions are satisfied.
\begin{itemize}
    \item[(i)] The function $y\mapsto \|f_y\|_p$ is non-increasing and continuous on $(0,\infty)$.
    \item[(ii)] $f(\cdot+i\varepsilon)$ is in $A^{p,q}_{\omega^{-k},\alpha}(\mathbb{C}_+)$ for any $\varepsilon>0$, and tends to $f$ in $A^{p,q}_{\omega^{-k},\alpha}(\mathbb{C}_+)$ as $\varepsilon$ tends to zero.
\end{itemize}
\end{proposition}
\begin{proof}
Assertion (ii) in Proposition \ref{prop:basicineq}, means that $f_y$ is in $H^p(\mathbb{C}_+)$, hence (i) holds.
\vskip .1cm
We proceed to prove (ii): it follows using (i) that
\begin{eqnarray*}\|f(\cdot+i\varepsilon)\|_{A^{p,q}_{\omega^{-k},\alpha}}^q &=& \int_0^\infty\|f_{y+\varepsilon}\|_p^q\,\omega^{- k}(y)y^\alpha dy\\ &\le& \int_0^\infty\|f_{y}\|_p^q\,\omega^{-k}(y)y^\alpha dy\\ &=& \|f\|_{A^{p,q}_{\omega^{- k},\alpha}}^q.\end{eqnarray*}
We have that $f(\cdot+i\varepsilon)$ has $f$ as limit as $\varepsilon\rightarrow 0$. Hence by (i) and \cite[Theorem 5.6]{st-book-fourier-analysis}, we have that $$\lim_{\varepsilon\rightarrow 0}\|f_{y+\varepsilon}-f_y\|_p=0$$ 
It follows from the above and the Dominated Convergence Theorem that 
$$\lim_{\varepsilon\rightarrow 0}\|f(\cdot+i\varepsilon)-f\|_{A^{p,q}_{\omega^{- k},\alpha}}^q=\lim_{\varepsilon\rightarrow 0}\int_0^\infty\|f_{y+\varepsilon}-f_y\|_p^q\,\omega^{- k}(y)y^\alpha dy=0.$$
\end{proof}
Let us prove the following.
\begin{proposition}\label{prop:densitypqrho}
Let $1\leq p,q<\infty$, and  $\alpha>-1$. Assume that $\Phi\in \mathscr{L}\cup\mathscr{U}$. Put $\omega=\mathcal{T}_\Phi^{\Vec{\varepsilon}}$, and let $k>0$. Then
$A^{p,q}_{\alpha}(\mathbb{C}_+)$ is a dense subspace of $A^{p,q}_{\omega^{-k},\alpha}(\mathbb{C}_+)$.
\end{proposition}
\begin{proof}
For $m>0$ a large enough integer, we have that for any $\varepsilon>0$, the function $g_\varepsilon(z)=(1-i\varepsilon z)^{-m}$ is in $A^{p,q}_{\alpha}(\mathbb{C}_+)$. 
\vskip .1cm
Let $f\in A^{p,q}_{\omega^{-k},\alpha}(\mathbb{C}_+)$, and define $F^{(\varepsilon)}$ by $$F^{(\varepsilon)}(z)=g_\varepsilon(z)f(z+i\varepsilon).$$
Then by using assertion (ii) in Proposition \ref{prop:densityomega}, one sees that $F^{(\varepsilon)}$ belongs to $A^{p,q}_{\omega^{-k},\alpha}(\mathbb{C}_+)$. Also if $\rho(t)=1+\frac{kp}{p(1+\alpha)+q}+\ln(e+t)$ (one can also use $\ln(e+\frac 1t)$ when $t$ is small), then as the function $y\mapsto (\rho(y))^{\frac {k}q}y^{-\frac{1+\alpha}{q}-\frac 1p}$ is non-increasing, one obtains from assertion (i) in Proposition \ref{prop:basicineq} that the factor $f(z+i\varepsilon)$ is bounded and hence that $F^{(\varepsilon)}$ belongs to $A^{p,q}_{\alpha}(\mathbb{C}_+)$.
\medskip

We have that $F^{(\varepsilon)}\rightarrow f$ as $\varepsilon\rightarrow 0$. The Dominated Convergence Theorem then leads to $$\lim_{\varepsilon\rightarrow 0}\|F^{(\varepsilon)}-f\|_{A^{p,q}_{\omega^{-k},\alpha}}=0.$$
\end{proof}
We will need the following.
\begin{proposition}\label{prop:boundednessP1rho}
Let $1\leq p,q<\infty$, and  $\alpha>-1$. Assume that $\Phi\in \mathscr{L}\cup\mathscr{U}$. Put $\omega=\mathcal{T}_\Phi^{\Vec{\varepsilon}}$, and let $k>0$. Let $\beta>-1$ with $1+\alpha<q(1+\beta)$. Then
the Bergman projection $P_\beta$ reproduces the functions in $A^{p,q}_{\omega^{-k},\alpha}(\mathbb{C}_+)$.
\end{proposition}
\begin{proof}
As $1+\alpha<q(1+\beta)$, $P_\beta$ reproduces the elements of $A^{p,q}_{\alpha}(\mathbb{C}_+)$, Proposition \ref{prop:densitypqrho} allows us to conclude that this projector also reproduces functions in $A^{p,q}_{\omega^{-k},\alpha}(\mathbb{C}_+)$.
\end{proof}
It follows from classical arguments and the above reproducing formula that the following holds.
\begin{proposition}\label{pro:duality}
Let $1< p,q<\infty$, and  $\alpha>-1$. Assume that $\Phi\in \mathscr{L}\cup\mathscr{U}$. Put $\omega=\mathcal{T}_\Phi^{\Vec{\varepsilon}}$, and let $k\in\mathbb{R}$. Then the dual space $\left(A^{p,q}_{\omega^{k},\alpha}(\mathbb{C}_+)\right)^*$ of the Bergman space $A^{p,q}_{\omega^{k},\alpha}(\mathbb{C}_+)$ identifies with $A^{p',q'}_{\omega^{k(1-q')},\alpha}(\mathbb{C}_+)$ under the duality pairing
$$\langle f,g\rangle_\alpha:=\int_{\mathbb{C}_+}f(z)\overline{g(z)}dV_\alpha(z), f\in A^{p,q}_{\omega^{k},\alpha}(\mathbb{C}_+), g\in A^{p',q'}_{\omega^{k(1-q')},\alpha}(\mathbb{C}_+).$$
\end{proposition}
\subsection{Derivative characterization of weighted mixed norm Bergman spaces}
In this section, we prove that $f\in A^{p,q}_{\omega^{k},\alpha}(\mathbb{C}_+)$ if and only if the function $x+iy\mapsto yf'(x+iy)$ belongs to $L^{p,q}_{\omega^{ k},\alpha}(\mathbb{C}_+)$.
\begin{theorem}\label{thm:derivcharact}
Let $1< p,q<\infty$, and  $\alpha>-1$. Assume that $\Phi\in \mathscr{L}\cup\mathscr{U}$. Put $\omega=\mathcal{T}_\Phi^{\Vec{\varepsilon}}$, and let $k\in\mathbb{R}$. Let $f$ be a holomorphic function on $\mathbb{C}_+$. Then the following assertions are equivalent.
\begin{itemize}
    \item[(a)]  $f$ belongs to $A^{p,q}_{\omega^{k},\alpha}(\mathbb{C}_+)$.
    \item[(b)] The function  $x+iy\mapsto yf'(x+iy)$ belongs to $L^{p,q}_{\omega^{k},\alpha}(\mathbb{C}_+)$.
\end{itemize}
\end{theorem}
\begin{proof}
(a)$\Rightarrow$(b): Recall the following inequality for the derivative $f'$ of the holomorphic function $f$ on the upper half-plane (see \cite[Problem 2]{BBGNPR})

$$y|f'(x+iy)|\leq \frac{C}{y^2}\int_{\frac y2<v<2y}\int_{|x-u|<y}|f(u+iv)|dudv.$$
Using Minkowski's inequality, we obtain
$$y\|f'(\cdot+iy)\|_p\le \left( \frac{C}{y}\int_{\frac y2<v<2y}\|f(\cdot+iv)\|_{p}^{p}dv\right)^{\frac{1}{p}}.$$
Since the function $v\mapsto \|f(\cdot+iv)\|_{p}$ is non-increasing, we deduce that
$$  y\|f'(\cdot+iy)\|_p\le C  \|f(\cdot+iy/2)\|_{p}. $$
It follows that
\begin{eqnarray*}
\int_0^\infty \|f'(\cdot+iy)\|_p^qy^{q+\alpha}\omega^{k}(y)dy &\leq& C\int_0^\infty\|f(\cdot+iy/2)\|_{p}^{q}y^{\alpha}\omega^{k}(y)dy\\  &\lesssim& \int_0^\infty\|f(\cdot+iy)\|_p^qy^{\alpha}\omega^{k}(y)dy.
\end{eqnarray*}
\medskip

(b)$\Rightarrow$(a): Using assertion (iii) of Proposition \ref{prop:basicineq}, one obtains that
$$|f(x+iy)|=\left|\int_y^\infty -f'(x+iv)dv\right|\le \int_y^\infty |f'(x+iv)|dv.$$
Using the above and Minkowski's inequality, we derive that
\begin{eqnarray*}
\|f(\cdot+iy)\|_p &\le& \left(\int_{\mathbb{R}}\left(\int_y^\infty |f'(x+iv)|dv\right)^pdx\right)^{1/p}\\ &\le& \int_y^\infty \|f'(\cdot+iv)\|_pdv\\ &\lesssim& \int_0^\infty\frac{ \|f'(\cdot+iv)\|_pv^{1+\alpha} dv}{(y+v)^{1+\alpha}}\\ &=& H_\alpha\left(v\|f'(\cdot+iv)\|_p\right)(y).
\end{eqnarray*}
The conclusion then follows from Theorem \ref{thm:main2}.
\end{proof}
\subsection{Atomic decomposition of functions in $A_{\omega^{k}}^{p,q}(\C_+)$}
Let us start by recalling the following decomposition of $\C_+$ which is essentially from \cite[Section 1]{Gonessa}.
\subsubsection{A Whitney decomposition of $\C_+$}
\medskip

The Bergman distance in the upper-half plane is given by
$$d_{Berg}(z,w):=\frac{1}{2}\log\left(\frac{1+\left|\frac{z-w}{\Bar{z}-w}\right|}{1-\left|\frac{z-w}{\Bar{z}-w}\right|}\right).$$
We denote by $B_\rho(z_0)$ the Bergman ball about $z_0$ with radius $\rho$, i.e
$$B_\rho(z_0):=\{z\in \C_+: d_{Berg}(z,z_0)<\rho\}.$$
Let  $0<\delta< 1$   and $\gamma$ a real such that
$$  \frac{\log\left( \frac{1+\frac{\delta^{2}}{20}}{1-\frac{\delta^{2}}{20}}  \right)}{4 \log 2}< \gamma  < \frac{\log\left( \frac{1+\frac{\delta^{2}}{4}}{1-\frac{\delta^{2}}{4}}  \right)}{4 \log 2}.    $$
For $l,j\in \mathbb{Z}$, put 
\begin{equation}\label{eq:lattice}
z_{l,j}= \frac{\delta^{2}}{4}l2^{\gamma j-1}+i2^{\gamma j}=x_{l,j}+iy_j,
\end{equation}
and define
\begin{align}
I_{l,j}&:= \left\{x\in \R: \left|2^{-\gamma j}(x-x_{l,j})\right|<\frac{\delta^2}{4}   \right\} \\
I_{l,j}'&:=\left\{x\in \R: \left|2^{-\gamma j}(x-x_{l,j})\right|<\frac{\delta^2}{20}   \right\} \\
J_{j}&:=  \left\{y\in (0,\infty): \left|2^{-\gamma j}(y-y_j)\right|<\frac{\delta^2}{4}   \right\} \\
J_{j}'&:=  \left\{y\in (0,\infty): \left|2^{-\gamma j}(y-y_j)\right|<\frac{\delta^2}{20}   \right\}. 
\end{align}



We have also the following properties.
\begin{lemma}\label{lem:intervals}
\begin{itemize}
Let $l,j\in\Z$. The following assertions are satisfied.
\item[(i)] $] 0; \infty [ =\cup_{j\in \mathbb{Z}}J_{j}$ and $\mathbb{R}=\cup_{l\in \mathbb{Z}}I_{l,j}$, for all $j$ fixed in $\mathbb{Z}$;
\item[(ii)] for fixed  $j\in \mathbb{Z}$, intervals $I_{l,j}'$ are pairwise disjoint;
\item[(iii)] the intervals $J_{j}'$ are pairwise disjoint;
\item[(iv)] for fixed  $j\in \mathbb{Z}$, each element  $x\in \mathbb{R}$ is in at most four intervals $I_{l,j}$;
\item[(v)] there exists an integer $N$ such that any $y \in ] 0; \infty [$ is in at most $N$ intervals $J_{j}$.
\end{itemize}
\end{lemma}
One can check (see \cite{Gonessa}) that
$$I_{l,j}+iJ_j\subset B_{\delta^2}(z_{l,j})$$
and $$B_{\frac{\delta^2}{80}}(z_{l,j})\subset I_{l,j}'+iJ_j'.$$
The sequence $\{z_{l,j}\}_{l,j\in\Z}$ above is called a $\delta$-lattice in $\C_+$.
\medskip

\subsubsection{Atomic decomposition}
\medskip

For $y>0$ and $x \in \mathbb{R}$, put 

\begin{equation}\label{eq:5n}
L=L(y):=\{ j\in \mathbb{Z}: y \in J_{j}   \}
\end{equation}
and 
\begin{equation}\label{eq:5naq}
L_{k}=L_{k}(x):=\{ l\in \mathbb{Z}: x \in I_{l,k}   \},
\end{equation}
for all $k\in \mathbb{Z}$.
\medskip

In what follows, we will denote by $\mathcal{I}(F)$ the following quantity

\begin{equation*}\label{eq:5naqaq}
\int_{0}^{\infty}\sum_{j\in \mathbb{Z}}\chi_{J_{j}}(y) \left( \int_{0}^{\infty} \int_{-\infty}^{+\infty} |F(u+iv)|^p  \left( \int_{-\infty}^{+\infty} \sum_{l\in \mathbb{Z}}\chi_{\{x\in I_{l,j}: |(x+iy)- (u+iv)| < \frac{y}{2\sqrt{2}}  \}}(x)dx  \right) \frac{du dv}{v^{2}}   \right)^{\frac qp} d\omega_{\alpha,k}(y),
\end{equation*}
where $d\omega_{\alpha,k}(y)=\omega^{k}(y)y^{\alpha}dy$.
\medskip

We have the following estimate.
\begin{lemma}\label{pro:mainppm6}
Let $\alpha>-1$, $1\leq p,q<\infty$. Assume that $\Phi\in \mathscr{L}\cup\mathscr{U}$. Put $\omega=\mathcal{T}_\Phi^{\Vec{\varepsilon}}$, and let $k\in\mathbb{R}$. For $F \in A^{\Phi}_{\omega^{k},\alpha}(\mathbb{C_{+}})$, 
\begin{equation}\label{eq:inegaleaqay}
\mathcal{I}(F) \leq C\|F\|_{A^{p,q}_{\omega^{k},\alpha}}^{q}.
 \end{equation}
\end{lemma} 
\begin{proof}
From the property (iii) in Lemma \ref{lem:intervals}, we have that     $$\sum_{l\in \mathbb{Z}}\chi_{\{x\in I_{l,j}: |(x+iy)- (u+iv)| < \frac{y}{2\sqrt{2}}  \}}\leq 4.$$
Hence
\begin{eqnarray*}
    S &:=& \int_{-\infty}^{+\infty} \sum_{l\in \mathbb{Z}}\chi_{\{x\in I_{l,j}: |(x+iy)- (u+iv)| < \frac{y}{2\sqrt{2}}  \}}(x)dx\\ &\leq& 4 \int_{\{x\in\R:\, |(x+iy)- (u+iv)| < \frac{y}{2\sqrt{2}}  \}} dx\\ &\leq& 4\chi_{\{|y-v|<\frac{y}{2\sqrt{2}}\}} (v)\int_{\{x\in\R:\, |x- u| < \frac{y}{2\sqrt{2}}  \}} dx\\ &\leq& Cv\chi_{\{|y-v|<\frac{y}{2\sqrt{2}}\}} (v)
\end{eqnarray*}
where we have used in the last inequality the fact that if $x+iy\in\C_+$ and $u+iv\in\C_+$ are such that 
$|(x+iy)- (u+iv)| < \frac{y}{2\sqrt{2}}$, then $y\simeq v$.
 It follows that, 
\begin{align*}
\mathcal{I}(F)&\lesssim\int_{0}^{\infty}\sum_{j\in \mathbb{Z}}\chi_{J_{j}}(y) \left( \int_{0}^{\infty} \int_{-\infty}^{+\infty} |F(u+iv)|^p  \left( v \chi_{\{y/C < v < Cy \}}(v)\right) \frac{du dv}{v^{2}}   \right)^{\frac qp} d\omega_{\alpha, k}(y)\\
&\lesssim\int_{0}^{\infty}\sum_{j\in \mathbb{Z}}\chi_{J_{j}}(y) \left( \int_{y/C}^{Cy} \int_{-\infty}^{+\infty} |F(u+iv)|^p  v \frac{du dv}{v^{2}}   \right)^{\frac qp} d\omega_{\alpha,k}(y)\\
&=\int_{0}^{\infty}\sum_{j\in \mathbb{Z}}\chi_{J_{j}}(y) \left( \int_{y/C}^{Cy} \|F(.+iv)\|_{L^p}^{p} \frac{ dv}{v}   \right)^{\frac qp} d\omega_{\alpha,k}(y).
\end{align*}
Given that $y\mapsto \|F(.+iy)\|_{L^p}$ is non-increasing on $(0,\infty)$, we have that $$  \int_{y/C}^{Cy} \|F(.+iv)\|_{L^p}^{p} \frac{dv}{v} \lesssim  \left\|F\left(.+i\frac{y}{C}\right)\right\|_{L^p}^{p}.  $$ 
Therefore, using property (v) of Lemma \ref{lem:intervals}, we obtain
\begin{align*}
\mathcal{I}(F)&\lesssim \int_{0}^{\infty}\sum_{j\in \mathbb{Z}}\chi_{J_{j}}(y) \left( \int_{y/C}^{Cy} \|F(.+iv)\|_{L^p}^{p} \frac{dv}{v}   \right)^{\frac qp} d\omega_{\alpha, k}(y) \\
&\lesssim  \int_{0}^{\infty}\sum_{j\in L}\chi_{J_{j}}(y) \left\|F\left(.+i\frac{y}{C}\right)\right\|_{L^p}^{q} d\omega_{\alpha,k}(y) \\
&\lesssim \int_{0}^{\infty} \left\|F\left(.+iy\right)\right\|_{L^p}^{q} d\omega_{\alpha,k}(y) \lesssim \|F\|_{A^{p,q}_{\omega^{k},\alpha}}^{q}.
\end{align*}
\end{proof}
The following result follows essentially from the Mean Value Theorem, we refer to \cite{Gonessa} for details.
\begin{lemma}\label{lem:stepsampling}
For $1\leq p<\infty$, there exists a constant $C>0$ such that for every holomorphic function $F$ on $\mathbb{C}_+$, and for every $\delta\in (0,1)$, the following hold.
  \begin{itemize}
      \item[(a)] $|F(z)|^p\leq C\delta^{-2}\int_{d_{Berg}(z,u+iv)<\delta}|F(u+iv)|^p\frac{dudv}{v^2}$.
      \item[(b)] If $d_{Berg}(z,\zeta)<\delta$ and $\delta$ is small enough, then $$|F(z)-F(\zeta)|^p\leq C\delta^p\int_{d_{Berg}(z,u+iv)<1}|F(u+iv)|^p\frac{dudv}{v^2}.$$
  \end{itemize}
\end{lemma}
Let us prove the following lemma.
\begin{lemma}\label{lem:estimate-p-norm}
  Suppose that $\delta\in (0,1)$, and let $1\leq p,q<\infty$. Then there exists a constant $C>0$  such that the following hold.
  \begin{equation}\label{eq:estimate-p-norm}
      \Vert F(\cdot+iy)\Vert_p^q\leq C\int_{\vert v-y \vert<y\frac{\delta^2}{4}}\Vert F(\cdot+iv)\Vert_p^q\frac{dv}{v}.
  \end{equation}
\end{lemma}
\begin{proof}
    The Mean Value Theorem gives us that
    $$\vert F(x+i)\vert^p\leq C_p\int_{\vert v-1\vert<\frac{\delta^2}{4}}\int_{\vert x-u\vert<\frac{\delta^2}{4}}\vert F(u+iv)\vert^pdudv.$$
    It follows that
    \begin{eqnarray*}
    \int_{\mathbb{R}}\vert F(x+i)\vert^p dx &\leq& C_p\int_{\vert v-1\vert<\frac{\delta^2}{4}}\int_{\vert s\vert<\frac{\delta^2}{4}}\int_{\mathbb{R}}\vert F(x+iv)\vert^pdxdsdv\\ &\leq& C_p \frac{\delta^2}{4}  \int_{\vert v-1\vert<\frac{\delta^2}{4}}\int_{\mathbb{R}}\vert F(x+iv)\vert^pdxdv.
    \end{eqnarray*}
   Using Jensen's inequality when $q<p$ or H\''older's inequality when $p\leq q$, we obtain  
   \begin{equation}\label{eq:estiminvarpoit}
   \Vert F(\cdot+i)\Vert_p^q\leq C_{p,q,\delta}\int_{\vert v-1\vert<\frac{\delta^2}{4}}\left(\int_{\mathbb{R}}\vert F(x+iv)\vert^pdx\right)^{\frac qp}dv.
   \end{equation}
   Let us apply (\ref{eq:estiminvarpoit}) to the function $F_y$ defined by $F_y(u+iv)=F(u+iyv)$. Using that for $v\in\{v>0: \vert v-y \vert<y\frac{\delta^2}{4}\}$, $v\simeq y$, we obtain
   \begin{eqnarray*}
       \Vert F(\cdot+iy)\Vert_p^q &=& \Vert F_y(\cdot+i)\Vert_p^q\\ &\lesssim& \int_{\vert v-1\vert<\frac{\delta^2}{4}}\left(\int_{\mathbb{R}}\vert F_y(x+iv)\vert^pdx\right)^{\frac qp}dv\\ &=& \int_{\vert v-1\vert<\frac{\delta^2}{4}}\left(\int_{\mathbb{R}}\vert F(x+iyv)\vert^pdx\right)^{\frac qp}dv\\ &\lesssim& \int_{\vert v-y \vert<y\frac{\delta^2}{4}}\Vert F(\cdot+iv)\Vert_p^q\frac{dv}{v}.
   \end{eqnarray*}
   The proof is complete.
\end{proof}
 \begin{theorem}\label{thm:sampling}
Let $\alpha>-1$, $0<\delta<1$ and  $\{z_{l,j}=\frac{\delta^{2}}{4}l2^{\gamma j-1}+i2^{\gamma j}\}_{l,j\in\Z}$ the associated $\delta$-lattice defined above. Let   $\alpha>-1$ and $1< p,q<\infty$. Assume that $\Phi\in \mathscr{L}\cup\mathscr{U}$. Put $\omega=\mathcal{T}_\Phi^{\Vec{\varepsilon}}$, and let $k\in\mathbb{R}$. There exists a $C_\delta=C(\delta,p,q)>0$ such that for every $F \in A^{p,q}_{\omega^{ k}\alpha}(\mathbb{C_{+}})$, 
\begin{equation}\label{eq:sampling1}
\sum_{j}\left(\sum_{l}|F(z_{l,j})|^p\right)^{\frac qp}\omega^{ k}(2^{\gamma j})2^{\gamma j(\alpha+1+\frac qp)}\leq C_\delta\|F\|_{A^{p,q}_{\omega^{ k},\alpha}}^q.
 \end{equation}
 Moreover, if $\delta$ is small enough, then the converse inequality holds, i.e
 \begin{equation}\label{eq:sampling2}
 \|F\|_{A^{p,q}_{\omega^{k},\alpha}}^q\leq C_\delta\sum_{j}\left(\sum_{l}|F(z_{l,j})|^p\right)^{\frac qp}\omega^{k}(2^{\gamma j})2^{\gamma j(\alpha+1+\frac qp)}.
 \end{equation}

\end{theorem} 
\begin{proof}  
From assertion (a) in Lemma \ref{lem:stepsampling} and the fact that $B_{\frac{\delta^2}{80}}(z_{l,j})\subset I_{l,j}'+iJ_j'$, we obtain that
\begin{eqnarray*}
    |F(z_{l,j})|^p &\lesssim& \delta^{-4}\int_{B_{\frac{\delta^2}{80}}(z_{l,j})}|F(u+iv)|^p\frac{dudv}{v^2}\\ &\leq& \delta^{-4}\int_{I_{l,j}'}du\int_{J_j'}|F(u+iv)|^p\frac{dv}{v^2}\\ &\lesssim& \frac{\delta^{-4}}{y_j^2}\int_{I_{l,j}'}du\int_{J_j'}|F(u+iv)|^pdv
\end{eqnarray*}
where we have used the fact that for $v\in J_j'$, $v\simeq y_j$. It follows that 
\begin{eqnarray*}
    \sum_{l\in \mathbb{Z}}|F(z_{l,j})|^p &\lesssim& \sum_{l\in \mathbb{Z}}\frac{\delta^{-4}}{y_j^2}\int_{I_{l,j}'}du\int_{J_j'}|F(u+iv)|^pdv\\ &\lesssim& \frac{\delta^{-4}}{y_j^2}\int_{\mathbb{R}}du\int_{J_j'}|F(u+iv)|^pdv\\ &=& \frac{\delta^{-4}}{y_j^2}\int_{J_j'}\|F(\cdot+iv)\|_p^pdv\\ &\leq& \frac{\delta^{-4}}{y_j^2}\int_{J_j'}\|F(\cdot+i\lambda y_j)\|_p^pdv\quad \left(\lambda y_j<v\right)\\  &\lesssim& \frac{\delta^{-4}}{y_j}\|F(\cdot+i\lambda y_j)\|_p^p.
\end{eqnarray*}
Hence using Lemma \ref{lem:estimate-p-norm}, we obtain
\begin{eqnarray*}
    \sum_{j\in\mathbb{Z}}\left(\sum_{l\in \mathbb{Z}}|F(z_{l,j})|^p\right)^\frac{q}{p}\omega^{ k}(2^{\gamma j})2^{\gamma j(\alpha+1+\frac qp)} &\leq& C_{\delta}\sum_{j\in\mathbb{Z}}\|F(\cdot+i\lambda y_j)\|_p^q\omega^k(y_j)y_j^{\alpha+1}\\ &\lesssim& C_{\delta}\sum_{j\in\mathbb{Z}}\int_{J_j'}\|F(\cdot+iy)\|_p^q\omega^k(y_j)y^\alpha dy\\
     &\lesssim& C_{\delta}\int_{0}^\infty\|F(\cdot+iy)\|_p^q\omega^k(y)y^\alpha dy.
\end{eqnarray*}
To prove the converse, we first note that
\begin{eqnarray*}
    \|F(\cdot+iy)\|_p^p &\leq& C_p\left\{\sum_{l\in \mathbb{Z}}\int_{I_{l,j}}|F(x+iy)-F(z_{l,j})|^pdx+\sum_{l\in \mathbb{Z}}\int_{I_{l,j}}|F(z_{l,j})|^pdx\right\}\\ &\leq& C_{p,\delta}\left\{\sum_{l\in \mathbb{Z}}\int_{I_{l,j}}|F(x+iy)-F(z_{l,j})|^pdx+\sum_{l\in \mathbb{Z}}|F(z_{l,j})|^py_j\right\}.
\end{eqnarray*}
Using assertion (b) in Lemma \ref{lem:stepsampling}, we obtain
\begin{eqnarray*}
  T(y) &:=&  \sum_{l\in \mathbb{Z}}\int_{I_{l,j}}|F(x+iy)-F(z_{l,j})|^pdx\\ &=& \sum_{l\in \mathbb{Z}}\int_{\mathbb{R}}|F(x+iy)-F(z_{l,j})|^p\chi_{I_{l,j}}(x)dx\\ &\leq& C\delta^p\int_0^\infty\int_{-\infty}^\infty |F(u+iv)|^p\left(\int_{\mathbb R}\sum_{l\in L_j}\chi_{x\in I_{l,j}: |u+iv-x-iy|<C_\delta y}(x)dx\right)\frac{dudv}{v^2}.
\end{eqnarray*}
It follows that
\begin{eqnarray*}
    \|F\|_{A^{p,q}_{\omega^{k},\alpha}}^q &\leq& Cq\int_0^\infty \sum_{j\in\mathbb{Z}}\chi_{J_j}(y)T(y)^{\frac qp}\omega^k(y)y^\alpha dy+\sum_{j\in\mathbb{Z}}\left(\sum_{l\in \mathbb{Z}}|F(z_{l,j})|^p\right)^{\frac qp}\omega^k(y_j)y_j^{(\alpha+1+\frac qp)}\\ &\lesssim& \delta^q\mathcal{I}(F)+\sum_{j\in\mathbb{Z}}\left(\sum_{l\in \mathbb{Z}}|F(z_{l,j})|^p\right)^{\frac qp}\omega^k(y_j)y_j^{(\alpha+1+\frac qp)}.
\end{eqnarray*}
Using Lemma \ref{pro:mainppm6}, we obtain
$$\|F\|_{A^{p,q}_{\omega^{k},\alpha}}^q\leq C\left(\delta^q\|F\|_{A^{p,q}_{\omega^{k},\alpha}}^q+\sum_{j\in\mathbb{Z}}\left(\sum_{l\in \mathbb{Z}}|F(z_{l,j})|^p\right)^{\frac qp}\omega^k(y_j)y_j^{(\alpha+1+\frac qp)}\right).$$
This yields the inequality (\ref{eq:sampling2}) for $\delta$ small enough.
\end{proof}
For $\alpha>-1$, $1\leq p,q<\infty$ and $k\in\mathbb{R}$, we define the complex sequences space $\ell_{\omega^{ k},\alpha}^{p,q}=\ell_{\omega^{k},\alpha}^{p,q}(\C_+)$ by 
$$\ell_{\omega^{k},\alpha}^{p,q}:=\left\{\lambda=\{\lambda_{l,j}\}_{l,j\in\Z}:\|\lambda\|_{\ell_{\omega^{k},\alpha}^{p,q}}=\left(\sum_{j\in\Z}\left(\sum_{l\in\Z}|\lambda_{l,j}|^p\right)^{\frac qp}\omega^{k}(2^{\gamma j})2^{\gamma j(\alpha+1+\frac{q}{p})}\right)^{\frac 1q}<\infty\right\}.$$
We refer to \cite{Gonessa} for the following.
\begin{lemma}\label{lem:dualsmallspaces}
Let   $\alpha>-1$ and $1< p,q<\infty$. Assume that $\Phi\in \mathscr{L}\cup\mathscr{U}$. Put $\omega=\mathcal{T}_\Phi^{\Vec{\varepsilon}}$, and let $k\in\mathbb{R}$.  The topological dual of $\ell_{\omega^{k},\alpha}^{p,q}(\C_+)$,   $\left(\ell_{\omega^{k},\alpha}^{p,q}(\C_+)\right)^{*}$ coincides with $\ell_{\omega^{k (1-q')},\alpha}^{p',q'}(\C_+)$, in the sense that, for $T \in \left(\ell_{\omega^{k},\alpha}^{p,q}(\C_+)\right)^{*}$, 
there is a unique  $\mu:=\{\mu_{l,j}\}_{l,j\in \mathbb{Z}}\in \ell_{\omega^{k(1-q')},\alpha}^{p',q'}(\C_+)$ such that for all  $\lambda:=\{\lambda_{l,j}\}_{l,j\in \mathbb{Z}}\in \ell_{\omega^{k},\alpha}^{p,q}(\C_+)$,  
\begin{equation}\label{eq:smallpairing}  T (\mu,\lambda) :=\langle \mu, \lambda\rangle=\sum_{j\in \mathbb{Z}}\sum_{l\in \mathbb{Z}}\mu_{j,l}\overline{\lambda_{j,l}}2^{j\gamma (\alpha +2)}. 
\end{equation}
   
\end{lemma} 

We have the following atomic decomposition of the weighted Bergman spaces in consideration.
\begin{theorem}\label{thm:atomicdecomp}
  Let   $\alpha>-1$, $1< p,q<\infty$, $0<\delta<1$ and  $\{z_{l,j}=\frac{\delta^{2}}{4}l2^{\gamma j-1}+i2^{\gamma j}\}_{l,j\in\Z}$ the associated $\delta$-lattice. Assume that $\Phi\in \mathscr{L}\cup\mathscr{U}$. Put $\omega=\mathcal{T}_\Phi^{\Vec{\varepsilon}}$, and let $k\in\mathbb{R}$.  Then the following assertions are satisfied.
  \begin{itemize}
      \item[(a)] For any sequence of complex numbers  $\{\lambda_{l,j}\}_{l,j\in\Z}$ such that $$\sum_{j\in\Z}\left(\sum_{l\in\Z}|\lambda_{l,j}|^p\right)^{\frac qp}\omega^{k}(2^{\gamma j})2^{\gamma j(\alpha+1+\frac{q}{p})}<\infty,$$
      the series $$\sum_{l,j}\lambda_{l,j}2^{\gamma j(\alpha+1+\frac qp)}K_\alpha(z,z_{l,j})$$
      is convergent in $A^{p,q}_{\omega^{k},\alpha}(\C_+)$. Moreover, its sum $F$ satisfies the estimate
 \begin{equation}\label{ineqatomdecomp}
  \|F\|_{A^{p,q}_{\omega^{k},\alpha}}^q\leq C_\delta\sum_{j}\left(\sum_{l}|F(z_{l,j})|^p\right)^{\frac qp}\omega^{k}(2^{\gamma j})2^{\gamma j(\alpha+1+\frac qp)}.   
 \end{equation}
 \item[(b)] If $\delta$ is small enough, then any function $F\in A^{p,q}_{\omega^{k},\alpha}(\C_+)$ may be represented as $$F(z)=\sum_{l,j}\lambda_{l,j}2^{\gamma j(\alpha+1+\frac qp)}K_\alpha(z,z_{l,j})$$
 with 
 \begin{equation}\label{eq:reverseineqatomdecomp}
   \sum_{j}\left(\sum_{l}|F(z_{l,j})|^p\right)^{\frac qp}\omega^{ k}(2^{\gamma j})2^{\gamma j(\alpha+1+\frac qp)}\leq C_\delta \|F\|_{A^{p,q}_{\omega^{k},\alpha}}^q.  
 \end{equation}
  \end{itemize}
  
\end{theorem}
\begin{proof}
The follows the now classical arguments in the literature (see for example \cite{BBGNPR}).
\smallskip

   We recall with Lemma \ref{pro:duality} that the dual space
 $\left(A_{\omega^k,\alpha}^{p,q}(\mathbb{C}_+)\right)^{*}$ of the Bergman space $A_{\omega^k,\alpha}^{p,q}(\mathbb{C}_+)$ identifies with
 $A_{\omega^{k(1-q')},\alpha}^{p',q'}(\mathbb{C}_+)$ under the integral pairing
 $$\langle f,g\rangle_{\alpha}=\int_{\mathbb{C}_+}f(z)\overline {g(z)}dV_{\alpha}(z).$$
Lemma \ref{lem:dualsmallspaces} gives us that
the dual space $(\ell_{\omega^k,\alpha}^{p,q}(\mathbb{C}))^{*}$ of $\ell_{\omega^k,\alpha}^{p,q}(\mathbb{C})$
identifies with $\ell_{\omega^{k(1-q')},\alpha}^{p',q'}(\mathbb{C})$ under the sum
pairing (\ref{eq:smallpairing}).

$(i)$ From the first part of the sampling theorem, Theorem \ref{thm:sampling}, we deduce that the linear operator
$$R:\,\,\,A_{\omega^k,\alpha}^{p,q}(\mathbb{C}_+)\rightarrow \ell_{\omega^k,\alpha}^{p,q}(\mathbb{C})$$ $${F\mapsto \{F(z_{l,j})\}}$$
 is bounded. Hence its  adjoint under
the pairing $\langle \cdot,\cdot\rangle_{\alpha}$\,\,$$R^{*}:\,\,\,\ell_{\omega^{k(1-q')},\alpha}^{p',q'}(\mathbb{C})\rightarrow A_{\omega^{k(1-q')},\alpha}^{p',q'}(\mathbb{C}_+)$$ is
also bounded. It is then enough to show that
$$R^{*}(\{\lambda_{l,j}\})= \sum_{l,j} {\lambda_{l,j}2^{\gamma j(\alpha+2)}K_{\alpha}(\cdot,
z_{l,j})}.$$
For $F\in
A_{\omega^k,\alpha}^{p,q}(\mathbb{C}_+)$ and $\{\lambda_{l,j}\}\in \ell_{\omega^k,\alpha}^{p,q}(\mathbb{C})$, we have 
\begin{eqnarray*}
 \langle RF,\{\lambda_{l,j}\}\rangle_{\alpha} &=&
\sum_{l,j} F(z_{l,j})\overline {\lambda_{l,j}}2^{j\gamma (\alpha +2)}\\ & =& \sum_{l,j} (P_{\alpha}F(z_{l,j}))\overline
{\lambda_{l,j}}2^{j\gamma (\alpha +2)}\\ &=& \sum_{l,j}
\int_{\mathbb{C}_+}K_{\alpha}(z_{l,j},w)F(w)dV_\alpha (w)\overline {\lambda_{l,j}}2^{j\gamma (\alpha +2)}\\ &=& \int_{\mathbb{C}_+}F(w)\overline{\left(\sum_{l,j}
{\lambda_{l,j}}K_{\alpha}(w,z_{l,j})2^{j\gamma (\alpha +2)}\right)}dV_\alpha(w)\\ &=&
\langle F,R^{*}(\{\lambda_{l,j}\})\rangle_{\alpha}.
\end{eqnarray*}
This conclude the first part of the theorem.

\vskip .2cm
 $(ii)$ From the second part of Theorem \ref{thm:sampling}, we have for $\delta$ small enough,
$$||F||_{A_{\omega^{k(1-q')},\alpha}^{p',q'}} \le C_{\delta}||\{F(z_{l,j})\}||_{\ell_{\omega^{k(1-q')},\alpha}^{p',q'}}. $$
Thus $R^{*}:\,\,\,\ell_{\omega^{k(1-q')},\alpha}^{p',q'}(\mathbb{C})\rightarrow A_{\omega^{k(1-q')},\alpha}^{p',q'}(\mathbb{C}_+)$ is
onto. Moreover, if we denote by $\mathcal {N}$ the subspace of
$\ell_{\omega^{k(1-q')},\alpha}^{p',q'}(\mathbb{C})$ consisting of all sequences
$\{\lambda_{l,j}\}_{l,j\in\mathbb{Z}}$ such that the sum $\sum_{l,j}
\lambda_{l,j}K_{\alpha}(z, z_{l,j})2^{\gamma j(\alpha+1+\frac qp)}$ is identically zero, then the linear map
$$\ell_{\omega^{k(1-q')},\alpha}^{p',q'}(\mathbb{C})\setminus\mathcal {N} \rightarrow A_{\omega^{k(1-q')},\alpha}^{p',q'}(\mathbb{C}_+)$$
is a bounded isomorphism. The continuity of its inverse which
follows from the Hahn-Banach theorem gives the estimate (\ref{eq:reverseineqatomdecomp}). This
completes the proof of the theorem. 

\end{proof}
\bibliographystyle{plain}

\end{document}